\documentclass[english,10pt]{amsart}

\title[On the relative Clemens conjectures]{The relative Clemens Conjectures for $\frac{1}{2}$-log Calabi-Yau threefolds }
\author{Rodolfo Aguilar}
\dedicatory{ \large With an appendix joint with \rm Adrian Zahariuc}
\date{}
\address{\parbox{\linewidth}{Centro de Investigacion en Matemáticas, A.C. (CIMAT), Jalisco S/N, Col. Valenciana CP: 36023 Guanajuato, Gto, México}}
\email{\href{mailto:aaguilar.rodolfo@gmail.com}{aaguilar.rodolfo@gmail.com}}
\urladdr{\url{https://sites.google.com/view/rodolfo-aguilar/}}

\usepackage[utf8]{inputenc} 
\usepackage[english]{babel}
\usepackage{amssymb}
\usepackage{amsthm}
\usepackage{mathrsfs}      
\usepackage{amsmath}
\usepackage{amsfonts}
\usepackage{hyperref}       

\theoremstyle{plain} 
\newtheorem{thm}{Theorem}[section]

\newtheorem{prop}[thm]{Proposition}
\newtheorem{cor}[thm]{Corollary} 

\newtheorem{conj}{Conjecture}[section]

\theoremstyle{definition} 
\newtheorem{ques}{Question}[section]

\DeclareMathOperator{\dAJ}{dAJ}
\DeclareMathOperator{\AJ}{AJ}
\DeclareMathOperator{\dP}{dP}
\DeclareMathOperator{\Pic}{Pic}

\newcommand{\ev}{\mathrm{ev}}
\newcommand{\abs}[1]{\left\vert#1\right\vert}
\newcommand{\Oo}{\mathscr{O}}
\newcommand{\Mbar}{\overline{\mathcal{M}}}
\newcommand{\Pbb}{\mathbb{P}}
\newcommand{\Ocal}{\mathcal{O}}

\begin{document}

\begin{abstract} 
We formulate a relative analogue of the Clemens conjectures for \(\tfrac12\)-log Calabi–Yau threefold pairs \((X,Y)\) (where \(K_X+2Y\cong\Oo_X\)). This framework rests on the restoration of a perfect deformation/obstruction duality specific to the \(\tfrac12\)-log CY threefold setting. Based on this duality, we conjecture that for a generic intersection configuration on the boundary divisor \(Y\), the number of rational curves anchored to these points is finite, and every such curve possesses the balanced relative normal bundle \(N_{C/X}(-Y)\cong \mathscr{O}_C(-1)\oplus\mathscr{O}_C(-1)\). In a joint appendix with Adrian Zahariuc, we verify this framework for prime Fano threefolds of index two. Using specialization techniques, we demonstrate that the usual virtual complications of relative Gromov–Witten theory are naturally suppressed in this setting. This trivialization of the relative moduli space cleanly reduces the virtual invariants to honest, classical enumerative counts, thereby rigorously proving the geometric conjecture.

\end{abstract}
      
\maketitle

\section{Introduction}
The Clemens conjectures \cite{C87} provide a precise predictive framework for the geometry of rational curves on quintic threefolds. These conjectures rest on a synthesis of two fundamental ingredients:
\begin{enumerate}
\item A bridge between Hodge theory (infinitesimal Abel–Jacobi maps) and deformation theory, and
\item A perfect duality between deformations and obstructions.
\end{enumerate}
In the absolute Calabi–Yau setting, this duality allows one to detect curves and, under suitable hypotheses, to count them. Nevertheless, the Clemens predictions do not extend verbatim to all non-rigid Calabi–Yau threefolds: Voisin \cite{V03} exhibited examples with infinite families of lines, showing that the naive enumerative expectations may fail.

It is natural to ask whether an analogous framework can be formulated in the log-smooth setting, for pairs $(X,Y)$. The first ingredient — the connection between log–Hodge theory and deformation theory — holds in considerable generality for $\mathbb{Q}$-log Calabi–Yau threefolds (see \cite{AGG24,A25}). The second ingredient, however, is more delicate: in the general log CY threefold setting (pairs $(X,Y)$ with $K_X+Y\sim0$) the deformation/obstruction duality breaks down, and the symmetry that underlies Clemens's arguments is lost.

In earlier work \cite{A25} I studied these structural questions in detail and showed that the duality is restored in the special ``$\frac{1}{2}$-log Calabi–Yau threefold'' case, namely when
$$K_X + 2Y \cong \mathscr{O}_X.$$
This $\frac{1}{2}$-log CY threefold duality is the technical input we exploit here. The purpose of the present note is to extract and organize its direct geometric consequences: what enumerative statements follow when the duality holds, and when can one reduce modern virtual enumerative problems to classical counting?

Concretely, we propose that the appropriate analogue of the Clemens phenomenon in the relative setting concerns ``anchored'' rational curves — rational curves together with a fixed intersection configuration on the boundary divisor $Y$. The precise prediction we study is the following.

\begin{conj}\label{conj:Int}
Let $(X,Y)$ be a smooth $\frac{1}{2}$-log CY threefold pair and let $C$ be a general smooth rational curve of $Y$-degree $e$ intersecting $Y$ transversally. Denote $Z=C\cap Y$. Then:
\begin{enumerate}
\item $X$ admits only a finite number of rational curves of degree $e$ whose intersection with $Y$ is exactly $Z$.
\item Moreover, all such curves have the relative normal bundle
$$N_{C/X}(-Y)\cong \mathscr{O}_C(-1)\oplus\mathscr{O}_C(-1).$$
\end{enumerate}
\end{conj}
To test these classical predictions using modern enumerative invariants, one faces a central difficulty: in the absolute Calabi-Yau setting (such as the quintic threefold), virtual Gromov-Witten invariants are notoriously difficult to relate to honest, classical curve counts due to the presence of highly degenerate maps and multiple covers. To overcome this, we emphasize a two-part methodological approach that clarifies the contribution of this note. First, Jun Li's relative Gromov-Witten formalism provides the conceptual dictionary, translating classical intersection constraints on the boundary divisor $Y$ into a precise fibre problem within a relative moduli space. Second, specialization methods developed by A.~Zahariuc in \cite{Z18} act as the computational engine, converting these abstract fibre-questions into finite, classical counts. By combining these tools, we demonstrate that in favorable $\frac{1}{2}$-log CY threefold geometries, virtual complications are naturally suppressed and the relative moduli spaces trivialize. Thus, relative GW theory serves as the bridge allowing modern machinery to cleanly output actual enumerative numbers, yielding a rigorous verification of the Clemens-type predictions that remain elusive in the absolute case.

Using these ideas, we obtain a rigorous verification for prime Fano threefolds of index two (also known as del Pezzo threefolds). Following the notation of \cite{Z18}, let $X=\dP_3^d$ be a smooth del Pezzo threefold of degree $d$. By definition, this is a smooth projective Fano threefold such that $-K_X=2Y$. We call the divisor $Y$ the hyperplane class. Note that $\dP_2^d:=Y$ is a del Pezzo surface of the same degree. 

Using the notation of relative Gromov-Witten invariants \cite{Li01} recalled in Appendix \ref{S:Virtual}, our main verification is the following.
    
\begin{thm}\label{Intthm:Main} Let $X$ be a smooth del Pezzo threefold of degree $d\in \{2,3,4,5,8\}$, $X$ general if $d=2$. Let $e \geq 1$, and let $\xi=(\xi_1,\ldots,\xi_e) \in Y^e$ be a configuration of \emph{general} points. Then the fiber of the relative evaluation morphism 
\[ \ev:\Mbar_\Gamma(X,Y) \longrightarrow Y^e \]
over $\xi$ contains no maps with nontrivial target expansions (no rubber levels) and no reducible domains. Furthermore, the natural identification induces an isomorphism of stacks
\[
\mathrm{ev}^{-1}(\xi) \;\xrightarrow{\ \cong\ }\; \Mbar_{0,e}(X,e)\times_{X^e}\{\xi\},
\]
and this fiber is finite, reduced, and consists exclusively of closed immersions with smooth source which intersect $Y$ transversally, and have balanced normal bundle $\Ocal(e-1) \oplus \Ocal(e-1)$. 
\end{thm}
    
Theorem \ref{Intthm:Main} simultaneously proves Conjecture \ref{conj:Int} for prime Fano threefolds of index two and demonstrates the suppression of virtual complications in this setting. The full proof is deferred to Appendix \ref{S:Virtual} (written jointly with A.~Zahariuc). The argument relies heavily on the specialization techniques developed in \cite{Z18} to establish the moduli space trivialization. Furthermore, the description of the twisted normal bundle follows from \cite[Thm 1.1]{Z18} and the irreducibility of the moduli space of degree $e$ rational curves not covering a line \cite{CS09,LT19}; see also \cite{BJ24}.

This outcome stands in sharp contrast with the absolute Calabi–Yau situation. While the formal ingredients appear in both contexts, the enumerative expectations that fail in the absolute case \cite{V03} are rigorously realized in the examples of prime Fano threefolds of index two treated here.

Finally, we remark on an intriguing broader context for these results. A key global ingredient in our proof of Theorem \ref{Intthm:Main} is the description of the irreducible components of the moduli space of rational curves on Fano threefolds (e.g., \cite{CS09, LT19}). The Geometric Manin's Conjecture, as formulated by Lehmann and Tanimoto \cite{LT19}, provides a predictive framework for the global asymptotic growth and distribution of these components. In contrast, our relative Clemens framework provides a localized, ``anchored'' perspective: it predicts exactly how these global moduli spaces trivialize when cut down by maximal incidence conditions on a specific boundary divisor. Exploring the precise interplay between the asymptotic global distribution of curves predicted by Geometric Manin and the exact, finite virtual counts of the relative Clemens conjectures presents a tantalizing avenue for future research.

The paper is organized as follows:
\begin{itemize}
\item In Section \ref{S:Abs} we review the reasoning behind the classical Clemens conjectures.
\item In Section \ref{S:RelClem} we extract the enumerative consequences of the $\frac{1}{2}$-log CY threefold duality and formulate precise conjectures about anchored rational curves.
\item In Appendix \ref{S:Virtual}, written jointly with Adrian Zahariuc, we explain how the relative GW perspective, the $\frac{1}{2}$-log CY threefold duality and the specialization techniques used in \cite{Z18} combine to reduce the virtual problem to classical counting in the Fano threefold of index two and degree $d\geq 2$.

\end{itemize}

\section*{Acknowledgement} 
I would like to thank Mark Green and Phillip Griffiths for useful discussions that inspired this work. I also thank Cristhian Garay and Pedro Luis del Ángel for many stimulating discussions and for the excellent working environment at CIMAT.

This work was supported by a postdoctoral fellowship, Estancias Posdoctorales por México 2025, from the Secretaría de Ciencia, Humanidades, Tecnología e Innovación (SECIHTI), Mexico.

\section{Absolute case}\label{S:Abs}
In his ICM talk \cite{C87}, Herb Clemens stated a series of questions and conjectures regarding curves on Calabi-Yau threefolds and the Abel-Jacobi maps. First, we state them and next, we explain how one may infer them.

\begin{ques}[I]\label{ques:I}
Let $X$ be a smooth projective threefold with trivial canonical bundle. Let $F$ be a family of immersed rational curves on $X$. If $\dim F > 0$, is the Abel-Jacobi mapping 
$$ \AJ: F \to J(X) $$ 
always nontrivial?
\end{ques}

\begin{ques}[I$^*$]\label{ques:I*}
Let $X$ be a smooth projective threefold with trivial canonical bundle. Let $f:\mathbb{P}^1\to X$ be a locally immersive map. Is the infinitesimal Abel-Jacobi map 
$$ dAJ: H^0(N_{f,X})\to H^1(\Omega_X^2)^* $$
always injective?
\end{ques}

\begin{conj}[II, Clemens]\label{ques:II}
The generic quintic threefold $X\subset \mathbb{P}^4$ admits only a finite number of rational curves of each degree. Each rational curve is a smoothly embedded $\mathbb{P}^1$ with normal bundle 
$$ \mathscr{O}(-1)\oplus \mathscr{O}(-1). $$
Furthermore, all the rational curves on $X$ are mutually disjoint.
\end{conj}

\begin{ques}[III]\label{ques:III}
Let $X\subseteq \mathbb{P}^4$ be a generic hypersurface of degree $\ge 6$ (so that $K_V$ is positive). Is every curve $C\subset X$ a surface section? At least, is every smooth curve $C$ a ``surface section to first order''? That is, does the short exact sequence
$$ 0 \to N_{C/X} \to N_{C/\mathbb{P}^4} \to N_{X/\mathbb{P}^4}|_C \to 0 $$
split?
\end{ques}

\subsection{Families of rational curves}\label{ss:famRatCur}
We now motivate the above conjectures.

Let $X$ be a smooth projective variety of dimension $2n-1$ and let $F\to \mathbb{P}^1$ be a rational family of codimension $n$ cycles. The map $\AJ:\mathbb{P}^1\to J^n(X)$ is trivial because $J^n(X)$ is a complex torus. 

Question \ref{ques:I} would have a negative answer if we could construct a rational family whose generic fibre is a smooth rational curve. Note that even constructing isolated examples of rational curves is not straightforward \cite{C83}.

Suppose we can only construct families parametrized by higher genus curves. We may expect that the variation is non-trivial, this is, it can be detected by the first derivative. This motivates Question \ref{ques:I*}. 

If $X=\{\sum_{i=0}^4 X_i^5\}$ is the Fermat quintic threefold, it is known that $X$ contains cones of lines parametrized by Fermat plane curves of degree $5$, see \cite{AK91}. Both questions \ref{ques:I} and \ref{ques:I*} have a positive answer in this case.

\subsection{Duality and deformations}
We can deduce conjecture \ref{ques:II} from question \ref{ques:I*}, by connecting deformation theory to Hodge theory. 

Let $X$ be a smooth Calabi-Yau threefold and let $C\subset X$ be a smooth curve. One of the main features of curves in Calabi-Yau threefolds is the following duality
\begin{equation}\label{eq:dualCY}
H^0(N_{C/X})\cong H^1(N_{C/X})^*.
\end{equation}
This can be deduced from Serre duality and the adjunction formula.

Let us write $\mathcal{I}_C$ for the ideal sheaf of $C$ and denote by $$T_X(-\log C)=\{\nu\in T_X\mid \nu(\mathcal{I}_C)=\mathcal{I}_C\},$$ the sheaf of vector fields preserving the ideal $\mathcal{I}_C$. Note that it is not locally free. We have the following short exact sequence of sheaves on $X$:

\begin{equation}
0\to T_X(-\log C) \to T_X \to N_{C/X}\to 0.
\end{equation}
Taking cohomology, we obtain:

\begin{equation}
\cdots H^0(N_{C/X})\to H^1(T_X(-\log C))\to H^1(T_X)\overset{\delta}\to H^1(N_{C/X})\to \cdots
\end{equation}
Using a volume form $\Omega \in H^0(K_X)$, we have that 
$$H^1(T_X)\cong H^1(\Omega_X^2). $$
Using this and equation (\ref{eq:dualCY}), we can write $\delta: H^1(\Omega_X^2)\to H^0(N_{C/X})^*$. A simple but important observation is the following. 
\begin{thm}\label{thm:deltaDual} The map $\delta$ is the dual of $\dAJ$.
\end{thm}
See \cite{A25} for a short proof due to Mark Green.

\begin{cor}\label{cor:1stDefdAJ} The first order deformations of a curve $C\subset X$ are controlled by $\dAJ^*$.
\end{cor}

As $X$ is Calabi-Yau, the adjoint formula reads:
$$K_C\cong K_X\otimes \det N_{C/X}\cong \det N_{C/X} $$
As $C$ is a rational curve, we have that $K_C\cong \mathscr{O}_{\mathbb{P}^1}(-2)$. Pulling back, we have $f^* N_{C/X}\cong \mathscr{O}_{\mathbb{P}^1}(a)\oplus \mathscr{O}_{\mathbb{P}^1}(b)$ by the Grothendieck-Birkhoff theorem with $a+b=-2$. 

If we assume that question \ref{ques:I*} has a positive answer, we obtain by Corollary \ref{cor:1stDefdAJ} that $\delta:H^1(T_X)\to H^1(N_{C/X})$ is surjective. If we also assume that our smooth curve $C\subset X=X_0$  moves with a generic deformation $X_t$ of $X$, we must have $H^0(N_{C/X})=H^1(N_{C/X})=0$. If $C$ is a rational curve, by the above decomposition we have that $N_{C/X}=\mathscr{O}_{\mathbb{P}^1}(-1)\oplus \mathscr{O}_{\mathbb{P}^1}(-1)$.

\subsection{Triviality of Abel-Jacobi}
For question \ref{ques:III}, it is known that not every curve $C\subset V$ is a surface section. For the second part, the relation between the non-splitness of the short exact sequence in Question \ref{ques:III} and non-triviality of the Abel-Jacobi map was found by Clemens in \cite{C89}.

An answer to this question can be deduced from the results in \cite{G89}.

\section{Relative setting} \label{S:RelClem}
In this section, we propose an extension of the Clemens conjectures to the $\frac{1}{2}$ log CY smooth setting $(X,Y)$. This is, $X$ is a threefold, $Y$ a divisor such that $K_X + 2Y \cong \mathscr{O}_X$ both smooth. Throughout this section, we assume that the curve $C$ intersects the divisor $Y$ transversally.

\subsection{Families of anchored rational curves}
Let $X$ be smooth projective and $Y\subset X$ be a smooth divisor.  We can define Abel-Jacobi maps for $X\setminus Y$ and for the pair $(X,Y)$, see \cite{KLMS06, KL07}. 

We can ask the analogous question of Section \ref{ss:famRatCur}, this is, can we construct a rational family of rational curves on $X$? 

If $X \subset \mathbb{P}^4$ is a smooth cubic threefold, for degrees $2\leq d \leq 5$, we do have rational families of rational curves. We will see below that the correct question is the following: 

\begin{ques} Let $(X,Y)$ be a smooth $\frac{1}{2}$-log CY threefold pair. Let $C\subset X$ a general smooth rational curve of $Y$-degree $d$. Consider $Z=C\cap Y$ and fix it. Can we construct rational families of smooth general rational curves of degree $d$ such that for every element $C_t$ of the family, we have: $C_t\cap Y=Z$, this is, a rational family preserving $Z$?
\end{ques}

In this setting, we can have analogues to the questions (I), (I*), (II), (III) of section \ref{S:Abs}.

\begin{ques}[I] Let $C\subset X$ be a general rational curve of $Y$-degree $d$. Let $F$ be a family of rational curves whose generic component $C_t$ is a general rational curve of $Y$-degree $d$ and such that $C_t\cap Y=C\cap Y$ for all $t\in F$. If $\dim F>0$, is the relative Abel-Jacobi map always nontrivial?
\end{ques}
In \cite{A25}, we have considered log versions of infinitesimal Abel-Jacobi maps for any smooth pair $(X,Y)$.
\begin{ques}[I*]\label{ques:logI*} Let $C\subset X$ be a smooth rational curve of $Y$-degree $d$. Is the infinitesimal Abel-Jacobi map
\begin{equation}\label{eq:relI*}
\dAJ: H^0(N_{C/X}(-Y))\to H^1(\Omega_X^2(\log Y))^*
\end{equation}
always injective? 
\end{ques}

In \cite[Section 4]{A25}, we generalize a criterion of Clemens \cite{C89} to establish a non-vanishing condition for (\ref{eq:relI*}), which we apply to the Fermat cubic threefold $X=\{\sum_{i=0}^4 X_i^3\}$. Specifically, we consider a cone of lines over a smooth elliptic curve $E$ with vertex $p$ and choose a hyperplane $Y$ passing through $p$. For any line $L$ in this cone, the normal bundle splits as $N_{L/X}\cong \mathscr{O}(1)\oplus \mathscr{O}(-1)$, implying that $\dim H^0(N_{L/X}(-Y))=1$. By combining the generalized criterion with explicit computations, we show that the infinitesimal Abel-Jacobi map $dAJ$ is non-zero, thereby proving its injectivity in this setting.

\subsection{1/2 Duality and deformations}
To recover the duality (\ref{eq:dualCY}), let us consider the following setting: let $X$ be a smooth projective threefold and assume there exists a smooth divisor $Y\subset X$ such that $K_X+2[Y]\cong \mathscr{O}_X$, we think of $Y\in \abs{-\frac{1}{2} K_X}$. The basic examples are smooth Fano threefolds of Picard rank one and index two, such as the cubic threefold.

The main duality in this case is as follows. 

\begin{thm}\label{thm:LogDuality} Let $(X,Y)$ be as above, consider a smooth curve $C\subset X$. Then 
$$ H^0(N_{C/X}(-Y))\cong H^1(N_{C/X}(-Y))^* .$$ 
\end{thm}
\begin{proof}
By Serre duality:
$$H^1(N_{C/X}(-Y))^*=H^0(K_C\otimes N_{C/X}^*(Y))$$
By adjunction:
$$K_C=K_X|_C\otimes \det N_{C/X} $$
Using $K_X|_C\cong \mathscr{O}_C(-2Y)$ and $\det N_{C/X}\otimes N_{C/X}^*\cong N_{C/X}$, we obtain the result.
\end{proof}

We learned this argument from Mark Green. A similar computation appeared in \cite{AGG24b}.

Let $\mathcal{I}_C, \mathcal{I}_Y$ denote the ideal sheaves of $C$ and $Y$ respectively. Denote by 
$$T_X(-\log C+Y)=\{\nu\in T_X\mid \nu (\mathcal{I}_C)=\mathcal{I}_C, \nu (\mathcal{I}_Y)=\mathcal{I}_Y \} $$
the sheaf of vector fields preserving the ideals $\mathcal{I}_C$ and $\mathcal{I}_Y$. It fits in the following exact sequence \footnote{The surjectivity can fail in the intersection of $C$ with $Y$ is not transverse.}
$$0\to T_X(-\log C+Y)\to T_X(-\log Y)\to N_{C/X}\to 0 $$
By twisting by $\mathscr{O}_X(-Y)$ and taking the associated long-exact sequence we obtain a map:
$$\delta: H^1(T_X(-\log Y)(-Y))\to H^1(N_{C/X}(-Y)). $$
Using a generator of $H^0(K_X(2Y))\cong H^0(K_X(\log Y)(Y))$, we obtain an identification:
$$H^1(T_X(-\log Y)(-Y))\cong H^1(\Omega_X^2(\log Y)).$$
Using duality from Theorem \ref{thm:LogDuality} and an argument as in Theorem \ref{thm:deltaDual}, we obtained in \cite{A25}.
\begin{thm} The map
$$\delta: H^1(T_X(-\log Y)(-Y))\to H^1(N_{C/X}(-Y)),$$
is dual to 
$$\dAJ: H^0(N_{C/X}(-Y))\to H^1(\Omega_X^2(\log Y))^* .$$ 
\end{thm}
Therefore, 
\begin{cor} The first deformations of a smooth curve $C\subset X$ fixing the points $C\cap Y$ are controlled by the log $\dAJ^*$.
\end{cor}

If we assume that Question \ref{ques:logI*} has a positive answer, by the above Corollary we obtain that $\delta$ is surjective. We have proved in \cite{A25}, that under the extra hypothesis that $X$ is Fano, the deformations controlled by $T_X(-\log Y)(-Y)$ are unobstructed, this is, $H^2(T_X(-\log Y)(-Y))=0$.

Now, if we assume $C\subset X$ rational, then $$N_{C/X}\cong \mathscr{O}_C(a_1)\oplus \mathscr{O}_C(a_2)$$ with $a_1+a_2=\deg(-K_X|_C)-2$. For a curve $C$ to move with a generic deformation $X_t$ of $X$ given by $H^1(T_X(-\log Y)(-Y))$, we must have that $H^1(N_{C/X}(-Y))=H^0(N_{C/X}(-Y))=0$. For this to hold, we must have:
$$N_{C/X}(-Y)\cong \mathscr{O}_C(-1)\oplus \mathscr{O}_C(-1). $$
Note that this implies $N_{C/X} = \mathcal{O}(d-1) \oplus \mathcal{O}(d-1)$ if $Y$ cuts $C$ in $d$ points.

The analogue of conjecture \ref{ques:II} is:
\begin{conj}\label{conjRC} Let $(X,Y)$ be a smooth $\frac{1}{2}$-log CY threefold pair and $C$ a smooth rational curve of $Y$-degree $d$, all three general. Denote $Z=C\cap Y$, then $X$ admits only a finite number of general rational curves whose intersection with $Y$ is $Z$. Moreover, they all have normal bundle
$$N_{C/X}(-Y)\cong \mathscr{O}_C(-1)\oplus \mathscr{O}_C(-1). $$ 
\end{conj}

The general smooth rational curve inside $\mathbb{P}^3$ is known to be perfectly balanced. Consider $C$ a smooth rational curve of degree $n$ in $\mathbb{P}^3$. Eisenbud-van de Ven \cite{EvdV81} proved that if $C$ is general then 
$$
N_{C/X}\cong \mathscr{O}_C(2n-1)\oplus \mathscr{O}_C(2n-1).
$$

\subsection{Triviality of the Abel-Jacobi map}
We have done the first steps towards the analogue of question \ref{ques:III} in \cite{A25} by proving a Noether-Lefschetz theorem for the universal family of open hypersurfaces. But the analogue results of \cite{G89} are not yet available in the log-setting.

\appendix

\section{Relative Gromov-Witten Invariants}\label{S:Virtual}
\begin{center}
    {\sc Rodolfo Aguilar and Adrian Zahariuc\footnote{University of Windsor, Department of Mathematics and Statistics, 401 Sunset Avenue, Windsor, ON, N9B 3P4, Canada. E-mail: \texttt{adrian.zahariuc@uwindsor.ca}}}
\end{center}
\vspace{0.4cm}

\subsection{Relative Stable Maps}

Let $(X, D)$ be a pair consisting of a smooth projective variety $X$ and a smooth divisor $D$. We fix discrete data $\Gamma = (g, n, \beta, \vec{\mu})$, where $g$ is the genus, $\beta \in H_2(X, \mathbb{Z})$ is the curve class, and $\vec{\mu} = (\mu_1, \dots, \mu_n)$ is a partition of the intersection number $D \cdot \beta$, representing the contact orders at the $n$ marked points.

Jun Li constructs the moduli space of relative stable maps, denoted $\Mbar_{\Gamma}(X, D)$ or $\Mbar_{g,n}(X, D, \beta, \vec{\mu})$. A point in this space is represented by a map
$$ f: (C, p_1, \dots, p_n) \longrightarrow X[k] $$
where $X[k]$ is a target expansion of $X$ (an ``accordion'' of $k$ rubber components $P = \mathbb{P}(N_{D/X} \oplus \mathcal{O}_D)$ attached to $X$ along $D$), satisfying a series of conditions, see \cite{Li01}.

There is a relative evaluation map to the divisor:
$$ \ev_D: \Mbar_{\Gamma}(X, D) \longrightarrow D^n, \quad [f] \longmapsto (f(p_1), \dots, f(p_n)). $$

\subsection{Prime Fano threefolds of index two} Following the notation of \cite{Z18}, let $X=\dP_3^d$ be a smooth del Pezzo threefold of degree $d$. By definition, this is a smooth projective Fano threefold such that $-K_X=2Y$. We call the divisor $Y$ the hyperplane class. Note that $\dP_2^d:=Y$ is a del Pezzo surface of the same degree. 

	The main goal of this appendix is to prove the following fact.

	\begin{thm}\label{thm:NoBubbles}
			Let $X$ be a smooth del Pezzo threefold of degree $d\in \{2,3,4,5,8\}$, $X$ general if $d=2$. Let
			\[ \ev:\Mbar_\Gamma(X,Y) \longrightarrow Y^e \]
			as at the beginning of Section \ref{S:RelClem}, and let $\xi=(\xi_1,\dots,\xi_e)\in Y^e$ be a configuration of $e$ \emph{general} points on $Y$.
\begin{enumerate}
  \item The fiber $\mathrm{ev}^{-1}(\xi)$ contains no maps with nontrivial target expansion (no rubber levels).
  \item The fiber $\mathrm{ev}^{-1}(\xi)$ contains no reducible domains.
  \item The natural identification of an unexpanded relative map with an absolute stable map
  induces an isomorphism of stacks
  \[
  \mathrm{ev}^{-1}(\xi) \;\xrightarrow{\ \cong\ }\; \Mbar_{0,e}(X,e)\times_{X^e}\{\xi\},
  \]
  and this fiber is finite, reduced, and consists of smooth maps meeting $Y$ transversely.
\end{enumerate}
		\end{thm}
		
		Theorem \ref{thm:NoBubbles} will turn out to be a simple corollary of the proposition below.
	
	\begin{prop}\label{prop:CubicZ} Let $X,Y,e$ as in Theorem \ref{thm:NoBubbles} and let $\xi_1,\ldots,\xi_e \in Y$ be general points on $Y$. Then, the fiber of the evaluation morphism $\ev: \Mbar_{0,e}(X,e) \to X^e$ over $\xi = (\xi_1,\ldots,\xi_e)$ is finite, reduced, and consists of closed immersions with smooth source which intersect $Y$ transversally, and have balanced normal bundle $\Ocal(e-1) \oplus \Ocal(e-1)$. 
	\end{prop}
	
	\begin{proof}
We prove the statement by induction on $e$. The base case $e=1$ corresponds to the fiber of the evaluation map $\text{ev}: \overline{\mathcal{M}}_{0,1}(X,1) \to X$ over a general point $\xi_1 \in Y$, which geometrically represents the set of lines in $X$ passing through $\xi_1$. The fact that this set is finite (e.g., consisting of 12 lines for $d=2$, 6 lines for $d=3$, etc.), reduced, and consists of unobstructed closed immersions with balanced normal bundle is well-known; see, for instance, \cite[Section 2.2]{Z18}. Assume that $e \geq 2$ and that the statement is true for all $e'<e$.

		Let $(f:C \to X,x_1,\ldots,x_e) \in \ev^{-1}(\xi)({\mathbb C})$. First, we show that $C$ is smooth. Assume by way of contradiction that $q \in C$ is a node, and let $C = C_1 \cup_q C_2$. For $i=1,2$, let $f_i = f|_{C_i}$, $A_i = \{ j \in [e] \mid x_j \in C_i\}$, $a_i = |A_i|$, and $e_i = \deg f_i^*\Ocal_X(Y)$. Note that
		\begin{equation}\label{eqn:equalsums} a_1+a_2 = e = e_1+e_2. \end{equation}
		We claim that $e_1,e_2 > 0$. Assume, for instance, that $e_2 = 0$. Since $C_2$ is connected, it is contracted to a point by $f_2$, so, in particular, $a_2 \leq 1$ since $\xi_1,\ldots,\xi_e$ are distinct, that is, there is at most one marked point on $C_2$. If $C_2 \cong \Pbb^1$, it is a destabilizing component for $f$, if $C_2 \not\cong \Pbb^1$, its dual graph has at least two leaves, hence at least one leaf which doesn't contain $q$, which is thus a destabilizing component for $f$ -- in either situation, the stability of the map is contradicted. Hence, $e_1,e_2>0$. 
		
		Assume that $e_i \leq a_i$ for some $i \in \{1,2\}$. Let $A'_i \subseteq A_i$ such that $|A'_i| = e_i$. By the inductive assumption with $e_i < e$ in the role of $e$ (note that $\{\xi_j \mid j \in A'_i\}$ are still $e_i$ \emph{general} points on $Y$), $C_i \cong \Pbb^1$ and $f_i:C_i \to X$ is a closed immersion of degree $e_i$ transverse to $Y$. Then, $(f_i(C_i) \cdot Y) = e_i$, but $\xi_j \in f(C_i) \cap Y$ for all $j \in A_i$, so, it is forced that $a_i = e_i$ and $f_i(C_i) \cap Y = \{\xi_j \mid j \in A_i\}$. In particular, neither of the strict inequalities $e_1<a_1$ and $e_2<a_2$ is possible, so $e_1 = a_1$ and $e_2 = a_2$ by \eqref{eqn:equalsums}, and everything above holds true for $i=1,2$. 
		
		Heuristically, the reason the scenario above is impossible is that $f(C_1)$ and $f(C_2)$ ought to be disjoint under the assumption that $\xi_1,\ldots,\xi_e$ are general. More formally, we proceed as follows to derive a contradiction. Let ${\mathcal B}$ be the stratum of $\Mbar_{0,e}(X,e)$ consisting of stable maps $(g:D \to X,y_1,\ldots,y_e)$ such that $D = D_1 \cup_r D_2$, with $\{ j \in [e] \mid y_j \in D_i \} = A_i$ and $\deg (g|_{D_i})^*\Ocal_X(Y) = e_i$ -- see \cite{BM96} for the general construction. By definition, $[f] \in {\mathcal B}({\mathbb C})$. Since the points $\xi_1,\ldots,\xi_e$ are assumed general, there exists an open substack 
		\[ {\mathcal Y} \subseteq {\mathcal B} \times_{\ev|_{\mathcal B}, X^e} Y^e \] 
		such that $[f] \in {\mathcal Y}({\mathbb C})$, any $(g:D \to X,y_1,\ldots,y_e) \in {\mathcal Y}({\mathbb C})$, with $D = D_1 \cup_r D_2$ as above, has the property that $D_i \cong {\mathbb P}^1$ and $g|_{D_i}$ is a closed immersion, and $\ev|_{\mathcal Y}$ is dominant onto $Y^e$. The fact that $\ev|_{\mathcal Y}$ is dominant implies that $\dim {\mathcal Y} \geq 2e$. Let
		\[ \varphi:\Mbar_{0,e}(X,e) \longrightarrow \Mbar_{0,0}(X,e) \]
		be the morphism which forgets the markings and stabilizes the resulting unmarked map. Note that $\varphi|_{\mathcal Y}$ is quasi-finite by the definition of ${\mathcal Y}$ -- indeed, the only freedom lies in permuting the indices of the marked points. However, this and the previous remark that $\dim {\mathcal Y} \geq 2e$ contradict the description of the irreducible components of $\Mbar_{0,0}(X,e)$ in \cite{LT19}. This concludes the proof of the fact that $C$ is smooth. 
		
		Second, we show that $f(C) \not\subset Y$. This can be proved in many ways; one way is to argue by specialization as in \cite{Z18}. Let $E \subset Y$ be a general hyperplane section of $Y$, and let us specialize $\xi_1,\ldots,\xi_e$ to general points $\xi_1,\ldots,\xi_e \in E$. Assume by way of contradiction that $f(C) \subset Y$. Since $\xi_1,\ldots,\xi_e \in f(C) \cap E$ and $(f_*[C] \cdot E)_Y = e$, we must have $f(C) \cap E = \{ \xi_1,\ldots,\xi_e \}$ scheme-theoretically and hence
		\[ \Ocal_Y(f(C))|_E \cong \Ocal_E(\xi_1+\cdots+\xi_e). \]
		However, this is impossible for general $\xi_1,\ldots,\xi_e \in E$, since $E$ is an elliptic curve, and the image of the restriction map $\Pic(Y) \to \Pic(E)$ is discrete. In conclusion, $f(C) \not\subset Y$. 
		
		Recall from \cite{LT19} that $\Mbar_{0,0}(X,e)$ has two irreducible components: the `main' component, which generically parametrizes closed immersions of $\Pbb^1$ into $X$, and one which generically parametrizes degree $e$ covers of lines in $X$. (Both components have dimension $2e$.) It is clear inductively that the same description applies to the irreducible components of $\Mbar_{0,n} (X,e)$, since $\Mbar_{0,n+1} (X,e) \to \Mbar_{0,n} (X,e)$ is flat of relative dimension $1$ with generically irreducible fibers, and that both components have dimension $2e+n$. 
				
		Let ${\mathcal V}$ be an irreducible component of $\Mbar_{0,e}(X,e) \times_{X^e} Y^e$ which dominates $Y^e$ by $\ev$. (We may endow ${\mathcal V}$ with the reduced structure.) Let 
		\[ \overline{\varphi}: \Mbar_{0,e}(X,e) \times_{X^e} Y^e \longrightarrow \Mbar_{0,0}(X,e) \] 
		be the restriction of $\varphi$. Note that any irreducible component of $\Mbar_{0,e}(X,e) \times_{X^e} Y^e$ has dimension at least $2e$, and
		\begin{equation}\label{eqn:dimlowerbound} 
			\dim {\mathcal V} \geq 2e
		\end{equation}
		in particular. Indeed, $\Mbar_{0,e}(X,e)$ has two irreducible components, both of dimension $3e$, and $\Mbar_{0,e}(X,e) \times_{X^e} Y^e$ can be described inside $\Mbar_{0,e}(X,e)$ as the common vanishing locus of $e$ sections of line bundles. Since ${\mathcal V}$ dominates $Y^e$, it follows from what was proven above that the general member of ${\mathcal V}$ has smooth source and is not contained in $Y$, and hence its image intersects $Y$ precisely in $\{\xi_1,\ldots,\xi_e\}$, the image of the marked points. Hence,
		\[ \overline{\varphi}|_{\mathcal V}: {\mathcal V} \longrightarrow \Mbar_{0,0}(X,e) \] 
		is generically finite. However, by \eqref{eqn:dimlowerbound}  and $\dim \Mbar_{0,0}(X,e) = 2e$, equality must occur in \eqref{eqn:dimlowerbound}, and $\overline{\varphi}|_{\mathcal V}$ must be surjective onto an irreducible component of the target, which can only be the main component. By \cite{LT19}, the general element of ${\mathcal V}$ must be a closed immersion. Since $\dim {\mathcal V} = 2e = \dim Y^e$ and $\ev|_{\mathcal V}: {\mathcal V} \to Y^e$ is surjective by assumption, $\ev|_{\mathcal V}$ must be generically finite. 
		
		Let $f:C \to X$ be a general element of the main irreducible component of $\Mbar_{0,0}(X,e)$ and $(f:C \to X,x_1,\ldots,x_e) \in {\mathcal V}({\mathbb C})$ a preimage of it by $\overline{\varphi}|_{\mathcal V}$. We have already shown that $C \cong \Pbb^1$, and that $f$ is a closed immersion which intersects $Y$ transversally at $\{f(x_1),\ldots,f(x_e)\}$. We claim that 
		\[ N_f \cong \Ocal(e-1) \oplus \Ocal(e-1). \]
		It suffices to prove that the main irreducible component of $\Mbar_{0,0}(X,e)$ contains an unramified stable map $f:{\mathbb P}^1 \to X$ with balanced normal. However, this follows immediately from Theorem 1.1 in \cite{Z18}. 
		
		The only remaining claim, that the general fiber of $\ev$ is reduced, follows from $N_f \cong \Ocal(e-1) \oplus \Ocal(e-1)$ by standard deformation theory. 
		\end{proof}

		\begin{proof}[Proof of Theorem \ref{thm:NoBubbles}]
			Let $(g:D \to W,q_1,\ldots,q_e) \in \ev^{-1}(\xi)({\mathbb C})$, where $\rho:W \to X$ is the corresponding target expansion. Then $(\rho \circ g:D \to X,q_1,\ldots,q_e)$ is a \emph{prestable} map into $X$. Its stabilization $(f:C \to X,x_1,\ldots,x_e)$ belongs to the fiber of the evaluation map $\Mbar_{0,e}(X,e) \to X^e$ over the general $\xi \in Y^e$. Applying Proposition \ref{prop:CubicZ}, it is a closed immersion of $C \cong \Pbb^1$ with balanced normal bundle, which intersects $Y$ precisely at $\{\xi_1,\ldots,\xi_e\}$, everywhere transversally. If $\rho$ is not an isomorphism, it is easy to see that the only way $g$ could be admissible is if $D$ is obtained from $C$ by attaching $e$ chains of lines at $x_1,\ldots,x_e$, which map isomorphically onto fibers of the ruled components of $W$. However, such a relative map is not stable, contradiction. Hence, $\rho$ is an isomorphism, and the everything else follows easily. 
		\end{proof}

\bibliographystyle{amsalpha}
\bibliography{sample}

@misc{A25,
      title={Infinitesimal invariants of mixed Hodge structures
 II: Log Clemens conjecture and log connectivity}, 
      author={Aguilar, Rodolfo},
      year={2026},
}

@article {AGG24b,
    AUTHOR = {Aguilar, Rodolfo and Green, Mark and Griffiths, Phillip},
     TITLE = {Lagrangian interpretation of {A}bel-{J}acobi mappings
              associated to {F}ano threefolds},
   JOURNAL = {Serdica Math. J.},
  FJOURNAL = {Serdica. Mathematical Journal. Serdika. Matematichesko
              Spisanie},
    VOLUME = {51},
      YEAR = {2025},
    NUMBER = {1},
     PAGES = {73--88},
      ISSN = {1310-6600,2815-5297},
   MRCLASS = {14J45 (14J28 14K30)},
  MRNUMBER = {4970676},
}

@misc{AGG24,
      title={Infinitesimal invariants of mixed Hodge structures}, 
      author={Rodolfo Aguilar and Mark Green and Phillip Griffiths},
      year={2024},
      eprint={2406.17118},
      archivePrefix={arXiv},
      primaryClass={math.AG},
      url={https://arxiv.org/abs/2406.17118}, 
}

@article {AK91,
    AUTHOR = {Albano, Alberto and Katz, Sheldon},
     TITLE = {Lines on the {F}ermat quintic threefold and the infinitesimal
              generalized {H}odge conjecture},
   JOURNAL = {Trans. Amer. Math. Soc.},
  FJOURNAL = {Transactions of the American Mathematical Society},
    VOLUME = {324},
      YEAR = {1991},
    NUMBER = {1},
     PAGES = {353--368},
      ISSN = {0002-9947,1088-6850},
   MRCLASS = {14J30 (14C30 14K30)},
  MRNUMBER = {1024767},
MRREVIEWER = {Fabio\ Bardelli},
       DOI = {10.2307/2001512},
       URL = {https://doi.org/10.2307/2001512},
}

@article {BM96,
    AUTHOR = {Behrend, K. and Manin, Yu.},
     TITLE = {Stacks of stable maps and {G}romov-{W}itten invariants},
   JOURNAL = {Duke Math. J.},
  FJOURNAL = {Duke Mathematical Journal},
    VOLUME = {85},
      YEAR = {1996},
    NUMBER = {1},
     PAGES = {1--60},
      ISSN = {0012-7094,1547-7398},
   MRCLASS = {14D20 (14C25 14D22)},
  MRNUMBER = {1412436},
MRREVIEWER = {Barbara\ Fantechi},
       DOI = {10.1215/S0012-7094-96-08501-4},
       URL = {https://doi.org/10.1215/S0012-7094-96-08501-4},
}

@misc{BJ24,
      title={Geometric Manin's Conjecture for Fano 3-Folds}, 
      author={Andrew Burke and Eric Jovinelly},
      year={2024},
      eprint={2209.05517},
      archivePrefix={arXiv},
      primaryClass={math.AG},
      url={https://arxiv.org/abs/2209.05517}, 
}

@article{C83,
 author = {Clemens, Herbert},
 title = {Homological equivalence, modulo algebraic equivalence, is not finitely generated},
 fjournal = {Publications Math{\'e}matiques},
 journal = {Publ. Math., Inst. Hautes {\'E}tud. Sci.},
 issn = {0073-8301},
 volume = {58},
 pages = {231--250},
 year = {1983},
 language = {English},
 url = {https://eudml.org/doc/103992},
 zbMATH = {3840112},
 Zbl = {0529.14002}
}

@misc{C87,
 author = {Clemens, Herbert},
 title = {Curves on higher-dimensional complex projective manifolds},
 year = {1987},
 language = {English},
 howpublished = {Proc. {Int}. {Congr}. {Math}., {Berkeley}/{Calif}. 1986, {Vol}. 1, 634-640 (1987).},
 zbMATH = {4116701},
 Zbl = {0682.14024}
}

@article {C89,
    AUTHOR = {Clemens, Herbert},
     TITLE = {The infinitesimal {A}bel-{J}acobi mapping and moving the
              {$\mathscr{O}(2)+\mathscr{O}(-4)$} curve},
   JOURNAL = {Duke Math. J.},
  FJOURNAL = {Duke Mathematical Journal},
    VOLUME = {59},
      YEAR = {1989},
    NUMBER = {1},
     PAGES = {233--240},
      ISSN = {0012-7094,1547-7398},
   MRCLASS = {14J30 (14C30)},
  MRNUMBER = {1016885},
MRREVIEWER = {John\ B.\ Little},
       DOI = {10.1215/S0012-7094-89-05907-3},
       URL = {https://doi.org/10.1215/S0012-7094-89-05907-3},
}

@article {CS09,
    AUTHOR = {Coskun, Izzet and Starr, Jason},
     TITLE = {Rational curves on smooth cubic hypersurfaces},
   JOURNAL = {Int. Math. Res. Not. IMRN},
  FJOURNAL = {International Mathematics Research Notices. IMRN},
      YEAR = {2009},
    NUMBER = {24},
     PAGES = {4626--4641},
      ISSN = {1073-7928,1687-0247},
   MRCLASS = {14H10},
  MRNUMBER = {2564370},
MRREVIEWER = {Scott\ R.\ Nollet},
       DOI = {10.1093/imrn/rnp102},
       URL = {https://doi.org/10.1093/imrn/rnp102},
}

@article{EvdV81,
 author = {Eisenbud, D. and van de Ven, A.},
 title = {On the normal bundles of smooth rational space curves},
 fjournal = {Mathematische Annalen},
 journal = {Math. Ann.},
 issn = {0025-5831},
 volume = {256},
 pages = {453--463},
 year = {1981},
 language = {English},
 doi = {10.1007/BF01450541},
 url = {https://eudml.org/doc/163526},
 zbMATH = {3691486},
 Zbl = {0443.14015}
}

@article {G89,
    AUTHOR = {Green, Mark L.},
     TITLE = {Griffiths' infinitesimal invariant and the {A}bel-{J}acobi
              map},
   JOURNAL = {J. Differential Geom.},
  FJOURNAL = {Journal of Differential Geometry},
    VOLUME = {29},
      YEAR = {1989},
    NUMBER = {3},
     PAGES = {545--555},
      ISSN = {0022-040X,1945-743X},
   MRCLASS = {14C30 (14K30)},
  MRNUMBER = {992330},
MRREVIEWER = {James\ D.\ Lewis},
       URL = {http://projecteuclid.org/euclid.jdg/1214443062},
}

@article{KLMS06,
 author = {Kerr, Matt and Lewis, James D. and M{\"u}ller-Stach, Stefan},
 title = {The {Abel}-{Jacobi} map for higher {Chow} groups},
 fjournal = {Compositio Mathematica},
 journal = {Compos. Math.},
 issn = {0010-437X},
 volume = {142},
 number = {2},
 pages = {374--396},
 year = {2006},
 language = {English},
 doi = {10.1112/S0010437X05001867},
 zbMATH = {5033726},
 Zbl = {1123.14006}
}

@article{KL07,
 author = {Kerr, Matt and Lewis, James D.},
 title = {The {Abel}-{Jacobi} map for higher {Chow} groups. {II}},
 fjournal = {Inventiones Mathematicae},
 journal = {Invent. Math.},
 issn = {0020-9910},
 volume = {170},
 number = {2},
 pages = {355--420},
 year = {2007},
 language = {English},
 doi = {10.1007/s00222-007-0066-x},
 zbMATH = {5208169},
 Zbl = {1139.14010}
}

@article {LT19,
    AUTHOR = {Lehmann, Brian and Tanimoto, Sho},
     TITLE = {Geometric {M}anin's conjecture and rational curves},
   JOURNAL = {Compos. Math.},
  FJOURNAL = {Compositio Mathematica},
    VOLUME = {155},
      YEAR = {2019},
    NUMBER = {5},
     PAGES = {833--862},
      ISSN = {0010-437X,1570-5846},
   MRCLASS = {14H10 (14J45)},
  MRNUMBER = {3937701},
MRREVIEWER = {Lidia\ Stoppino},
       DOI = {10.1112/s0010437x19007103},
       URL = {https://doi.org/10.1112/s0010437x19007103},
}

@article {Li01,
    AUTHOR = {Li, Jun},
     TITLE = {Stable morphisms to singular schemes and relative stable
              morphisms},
   JOURNAL = {J. Differential Geom.},
  FJOURNAL = {Journal of Differential Geometry},
    VOLUME = {57},
      YEAR = {2001},
    NUMBER = {3},
     PAGES = {509--578},
      ISSN = {0022-040X,1945-743X},
   MRCLASS = {14N35},
  MRNUMBER = {1882667},
MRREVIEWER = {Andreas\ Gathmann},
       URL = {http://projecteuclid.org/euclid.jdg/1090348132},
}

@incollection {V03,
    AUTHOR = {Voisin, Claire},
     TITLE = {On some problems of {K}obayashi and {L}ang; algebraic
              approaches},
 BOOKTITLE = {Current developments in mathematics, 2003},
     PAGES = {53--125},
 PUBLISHER = {Int. Press, Somerville, MA},
      YEAR = {2003},
      ISBN = {1-57146-103-5},
   MRCLASS = {32Q45 (14J99 32F45)},
  MRNUMBER = {2132645},
MRREVIEWER = {William\ A.\ Cherry},
}

@article {Z18,
    AUTHOR = {Zahariuc, Adrian},
     TITLE = {Rational curves on del {P}ezzo manifolds},
   JOURNAL = {Adv. Geom.},
  FJOURNAL = {Advances in Geometry},
    VOLUME = {18},
      YEAR = {2018},
    NUMBER = {4},
     PAGES = {451--465},
      ISSN = {1615-715X,1615-7168},
   MRCLASS = {14N10 (14D05 14H10 14J26 14N25 14N35)},
  MRNUMBER = {3871408},
MRREVIEWER = {Ragni\ Piene},
       DOI = {10.1515/advgeom-2018-0010},
       URL = {https://doi.org/10.1515/advgeom-2018-0010},
}
\end{document}